\documentclass[12pt,a4paper]{article}
\pdfoutput=1

\usepackage[T1]{fontenc}
\usepackage[utf8]{inputenc}
\usepackage[UKenglish]{isodate}
\usepackage{array}
\cleanlookdateon
\usepackage{amsfonts,amsmath,amssymb,mathtools,dsfont,microtype} 
\usepackage[dvipsnames]{xcolor} 
\usepackage{pdfcolfoot} 
\usepackage[noBBpl]{mathpazo} 

\usepackage[labelsep=period,labelfont=bf,justification=centering]{caption}
\usepackage{float,graphicx,subcaption}
\usepackage{booktabs,makecell}
\usepackage{enumitem,mdwlist}
\setlist[itemize]{topsep=0ex,itemsep=0ex,parsep=0.4ex}
\setlist[enumerate]{topsep=0ex,itemsep=0ex,parsep=0.4ex}

\usepackage{parskip,fullpage}
\usepackage{standalone}

\usepackage[hyphens]{url} 
\usepackage{xurl}
\usepackage[linktoc=all,hidelinks,colorlinks,unicode=true,hypertexnames=false]{hyperref} 
\usepackage[capitalise,compress,nameinlink,noabbrev]{cleveref} 

\hypersetup{linkcolor={blue!70!black}, citecolor={green!70!black}, urlcolor={blue!70!black}}

\usepackage{amsthm,thmtools,thm-restate}

\makeatletter
\@ifundefined{newcounteralias}{}{%
  \renewcommand\thmt@autorefsetup{%
    \@xa\def
    \csname\thmt@envname autorefname\@xa\endcsname
    \@xa{\thmt@thmname}%
  }%
}
\makeatother

\declaretheorem[name = Theorem, numberwithin = section, style = plain]{theorem}

\declaretheorem[name = Lemma, numberlike = theorem, style = plain]{lemma}

\declaretheorem[name = Proposition, numberlike = theorem, style = plain]{proposition}
\declaretheorem[name = Question, numberlike = theorem, style = plain]{question}

\crefname{definition}{Definition}{Definitions}
\crefname{theorem}{Theorem}{Theorems}
\crefname{lemma}{Lemma}{Lemmas}
\crefname{conjecture}{Conjecture}{Conjectures}
\crefname{claim}{Claim}{Claims}
\crefname{corollary}{Corollary}{Corollaries}
\crefname{observation}{Observation}{Observations}
\crefname{proposition}{Proposition}{Propositions}
\crefname{question}{Question}{Questions}
\crefname{remark}{Remark}{Remarks}
\crefname{section}{Section}{Sections}
\crefname{subsection}{\S}{\S\S}

\newenvironment{poc}{\begin{proof}}{\end{proof}}

\DeclarePairedDelimiter{\abs}{\lvert}{\rvert}

\DeclarePairedDelimiter{\set}{\lbrace}{\rbrace}
\DeclarePairedDelimiter{\floor}{\lfloor}{\rfloor}

\newcommand*{\bE}{\mathbb{E}}

\newcommand*{\bP}{\mathbb{P}}

\newcommand*{\cC}{\mathcal{C}}

\newcommand*{\cF}{\mathcal{F}}

\newcommand*{\cO}{\mathcal{O}}

\renewcommand{\epsilon}{\varepsilon}

\renewcommand{\geq}{\geqslant}
\renewcommand{\leq}{\leqslant}
\renewcommand{\preceq}{\preccurlyeq}
\renewcommand{\succeq}{\succcurlyeq}
\renewcommand{\emptyset}{\varnothing}
\renewcommand{\subset}{\subseteq}

\newcommand{\defn}[1]{{\color{Maroon}\ifmmode#1\else\emph{#1}\fi}}
\newcommand*{\eps}{\varepsilon}
\newcommand*{\yes}{\checkmark}
\newcommand*{\no}{\times}
\newcommand*{\newno}{\textcolor{blue}{\boldsymbol{\times}}}

\DeclareMathOperator{\child}{child}
\DeclareMathOperator{\dist}{dist}

\title{Hypergraphs with the smallest Hall ratio\\are still hard to colour}

\date{\today}
\author{Veronica Bitonti\footnotemark[1] \qquad  Freddie Illingworth\footnotemark[2] \\ Pedram Noohi\footnotemark[3] \qquad Jane Tan\footnotemark[4] 
}

\begin{document}
\maketitle
\begin{abstract}
For $H$ a hypergraph, the Hall ratio $\rho(H)=\text{max}_{X\subseteq V(H)}\frac{\abs{X}}{\alpha(H[X])}$ is a natural lower bound for its chromatic number, and among non-empty $k$-uniform hypergraphs the smallest possible Hall ratio is $k/(k-1)$. For graphs ($k = 2$), having minimum Hall ratio $2$ forces chromatic number $2$. Ne\v{s}et\v{r}il, R\"{o}dl, and Sales asked whether a similar phenomenon holds at higher uniformities. We answer this question negatively in a strong form, showing that for every $k \geqslant 3$ there are $k$-uniform hypergraphs with Hall ratio exactly $k/(k-1)$ and arbitrarily large fractional chromatic number.
\end{abstract}

\renewcommand{\thefootnote}{\fnsymbol{footnote}} 

\footnotetext[1]{Mathematical Institute, University of Oxford, UK (\textsf{\href{mailto:veronica.bitonti@maths.ox.ac.uk}{veronica.bitonti@maths.ox.ac.uk}})}

\footnotetext[2]{Department of Mathematics, University College London, UK (\textsf{\href{mailto:f.illingworth@ucl.ac.uk}{f.illingworth@ucl.ac.uk}})}

\footnotetext[3]{Department of Mathematics, London School of Economics, UK (\textsf{\href{mailto:p.noohi@lse.ac.uk}{p.noohi@lse.ac.uk}})}

\footnotetext[4]{All Souls College and Mathematical Institute, University of Oxford, UK (\textsf{\href{mailto:jane.tan@all-souls.ox.ac.uk}{jane.tan@all-souls.ox.ac.uk}})}

\renewcommand{\thefootnote}{\arabic{footnote}} 

\section{Introduction}
For a hypergraph $H$, we write $\alpha(H)$ for its independence number and $\chi(H)$ for its chromatic number. A classical lower bound for the chromatic number of $H$ is given by its \defn{Hall ratio}, 
\[
    \rho(H) \coloneqq \max\set[\Big]{\tfrac{\abs{X}}{\alpha(H[X])} \colon \emptyset \neq X \subset V(H)},
\]
where $H[X]$ is the sub-hypergraph induced by $X$. A natural question arises: can $\chi(H)$ be bounded by a function of $\rho(H)$?

The answer is best understood in the context of a third parameter, fractional chromatic number, an LP relaxation of chromatic number. Concretely, the \defn{fractional chromatic number}, $\chi_f(H)$, is the smallest $t > 0$ for which there is a random independent set $\mathbold{I}$ of $H$ such that $\bP(v \in \mathbold{I}) \geq 1/t$ for every vertex $v$.

The fractional chromatic number is always at most the chromatic number, as can be seen by taking $\mathbold{I}$ to be a uniformly random colour class in an optimal colouring of $H$. 
Next, taking $\mathbold{I}$ to be a random independent set of $H$ such that $\bP(v \in \mathbold{I}) \geq 1/\chi_f(H)$ for every $v$ gives $\alpha(H[X]) \geq \bE(\abs{\mathbold{I}\cap X}) = \sum_{v \in X} \bP(v \in \mathbold{I}) \geq \abs{X}/\chi_f(H)$ for every $X \subseteq V(H)$, and so there is a chain of general inequalities
\[
\rho(H) \leq \chi_f(H) \leq \chi(H).
\]
These parameters form a natural hierarchy of increasing complexity, and 
how far apart the parameters in this chain can be has been a topic of prolonged study. 

At the top of the chain, the corresponding question for graphs has a long and well-understood history.
Some of the earliest examples of graphs with bounded fractional chromatic number and arbitrarily large chromatic number are the shift graphs of Erd\H{o}s and Hajnal~\cite{EH64}, which satisfy $\chi_f\leq 4$.
Kneser graphs, whose chromatic number was first determined by Lov\'{a}sz~\cite{L78} (see also short proofs of B\'{a}r\'{a}ny~\cite{B78} and Greene~\cite{G02}), provide examples of graphs whose fractional chromatic number can be arbitrarily close to 2 and chromatic number can be arbitrarily large.

Separating $\rho$ from $\chi_f$ has been much harder. Graphs with bounded Hall ratio have bounded clique number, and for many examples of graphs with large chromatic number and bounded clique number, the Hall ratio provides a tight (or almost tight) lower bound for the fractional chromatic number. This motivated conjectures that the fractional chromatic number is bounded by a function of the Hall ratio. The most enticing of these, made by Harris, was that $\chi_f = \cO(\rho)$. This was subsequently disproved by Blumenthal, Lidick\'{y}, Martin, Norin, Pfender, and Volec~\cite{BLMNPV22}. Dvo\v{r}\'{a}k, Ossona de Mendez, and Wu~\cite{DOW20} obtained absolute separation: graphs with Hall ratio at most 18 can have arbitrarily large fractional chromatic number. The bound 18 was improved to $3 + \epsilon$ by Steiner~\cite{S25}, with further announced improvements by Rigaud and Walczak to $3$ using Burling graphs, and by Davies, Hatzel, and Yepremyan~\cite{DHY} to $2 + \epsilon$.

This sequence of improvements cannot be pushed all the way to 2, since this is the smallest Hall ratio of non-empty graphs and at this minimum a threshold behaviour occurs: having $\rho \leq 2$, $\chi_f \leq 2$, and $\chi \leq 2$ are all equivalent for graphs since they are equivalent to having no odd cycles.

For non-empty $k$-uniform hypergraphs, the smallest possible Hall ratio (and fractional chromatic number) is $k/(k - 1)$. Here, too, there are Kneser hypergraphs with fractional chromatic number arbitrarily close to $k/(k - 1)$ and arbitrarily large chromatic number (see Azarpendar and Jafari~\cite{AJ21}). Ne\v{s}et\v{r}il, R\"{o}dl, and Sales~\cite[Thm.~5.2]{NRS24} showed that there is again a threshold behaviour at $k/(k - 1)$ for fractional chromatic number: every $k$-uniform hypergraph with fractional chromatic number at most $k/(k - 1)$ has chromatic number at most $2$ (we give our own very short proof of this, see \cref{thm:chif-chi}). 

This leaves open the question of whether there is still a threshold behaviour at the minimum Hall ratio in hypergraphs with higher uniformities: 

\begin{question}[Ne\v{s}et\v{r}il, R\"{o}dl, and Sales~{\cite[Problem~6.2]{NRS24}}] 
Does there exist an integer $m(k)$ such that every $k$-uniform hypergraph $H$ with Hall ratio at most $k/(k - 1)$ has chromatic number at most $m(k)$\textup{?}
\end{question}
Our main result shows that, unlike for graphs, such an $m(k)$ does not exist for any $k \geq 3$. In fact, even the fractional chromatic number may be arbitrarily large.

\begin{theorem}\label{thm:main}
    For every $k \geq 3$, the \textup{(}fractional\textup{)} chromatic number of $k$-uniform hypergraphs with Hall ratio $k/(k - 1)$ is unbounded.
\end{theorem}

With this theorem, we now have a complete picture of the bounded and unbounded gaps between parameters in the chain as summarised in \cref{tab:summary}.

\begin{table}[ht]
\centering
\renewcommand{\arraystretch}{1.35}
\begin{tabular}{@{}cc @{\hspace{2em}} cc @{\hspace{2em}} cc@{}}
 &
 & \multicolumn{2}{c}{\textbf{Graphs} ($k=2$)$\qquad$}
 & \multicolumn{2}{c}{\textbf{Hypergraphs} ($k\geqslant 3$)} \\
\textbf{Is\ldots} & \textbf{bounded by\ldots?}
 & at min. & above min. & at min. & above min. \\
$\chi$   & $m(\chi_f)$ & $\yes$ & $\no$ & $\yes$   & $\no$ \\
$\chi_f$ & $m(\rho)$   & $\yes$ & $\no$ & $\newno$ & $\no$ \\
$\chi$   & $m(\rho)$   & $\yes$ & $\no$ & $\newno$ & $\no$ \\
\end{tabular}
\caption{Bounded and unbounded gaps between parameters in the chain, where $m$ is used to denote arbitrary functions. The blue entries are new, given by \cref{thm:main}, and are the only entries that differ between graphs and hypergraphs.}
\label{tab:summary}
\end{table}

Our constructions are built by `hanging' a hypergraph on a rooted tree $T$. 
Rooted tree constructions in hypergraph colouring have appeared previously, for example in work of Pach, Tardos, and T\'{o}th~\cite{PTT09}, Ackerman, Keszegh, and P\'{a}lv\"{o}lgyi~\cite{AKP20}, and Alon, Kostochka, Reiniger, West, and Zhu~\cite{AKRWZ16}. 
In those constructions every edge is contained either in a set of siblings or in a root-leaf path.
Here we use a different edge structure: the vertex-set of the hypergraph $H$ we construct is $V(T)$, and every edge consists of a vertex $u$ together with $k - 1$ children of some vertex $v$ that is $u$ or some descendant of $u$.

The hanging construction allows considerable freedom. 
The sub-hypergraph of edges consisting of $k - 1$ children and one ancestor of $v$ can be chosen independently for each vertex $v \in V(T)$. 
Furthermore, the small Hall ratio and large fractional chromatic number of $H$ are guaranteed by two distinct mild conditions on these sub-hypergraphs (roughly, large girth and small independence ratio respectively). 
We envisage this flexibility being useful for constructing hypergraphs that have large chromatic number/small Hall ratio as well as further prescribed properties.

\Cref{sec:hung} details the construction. 
The lower bound on the fractional chromatic number (\cref{lem:hung-chif}) is a probabilistic argument using a random root-leaf path in $T$. The main work lies in determining the Hall ratio, which is stated as \cref{lem:hung-Hall} and proved in \cref{sec:Hall}. The later sections use preliminary concepts that we introduce in \cref{sec:hyp}.

\section{Hypergraphs}\label{sec:hyp}

A \defn{hypergraph} $H$ consists of a vertex-set $V(H)$ and a set $E(H)$ of subsets of $V(H)$, called (hyper)edges. $H$ is \defn{$k$-uniform} if every edge has size $k$, and \defn{non-empty} if $E(H) \neq \varnothing$. A \defn{multi-hypergraph} is defined similarly, except that edges may appear with multiplicity greater than 1.
For a collection $\cF$ of edges, we define $\defn{V(\cF)} \coloneqq \bigcup_{e \in \cF} e$ to be the set of vertices incident to edges of $\cF$. A hypergraph $H'$ is a \defn{sub-hypergraph} of $H$ if $V(H')\subseteq V(H)$ and $E(H')\subseteq E(H)$, and for $X\subseteq V(H)$ we write $H[X]$ for the sub-hypergraph of $H$ induced by $X$. 

As mentioned in the introduction, Ne\v{s}et\v{r}il, R\"{o}dl, and Sales~\cite[Thm.~5.2]{NRS24} showed that every $k$-uniform hypergraph with fractional chromatic number at most $k/(k - 1)$ is 2-colourable. What follows is our own proof of this.

\begin{theorem}\label{thm:chif-chi}
    For $k \geq 2$, every $k$-uniform hypergraph with fractional chromatic number at most $k/(k - 1)$ has chromatic number at most $2$.
\end{theorem}

\begin{proof}
    Let $H$ be a hypergraph with $\chi_f(H) \leq k/(k - 1)$. Then there is a random independent set $\mathbold{I}$ of $H$ such that, for each vertex $v \in V(H)$, $\bP(v \in \mathbold{I}) \geq (k - 1)/k$. If $H$ has no edges, then it is 1-colourable. Otherwise, arbitrarily fix an edge $e \in E(H)$. We have
    \[
    \bE\bigl(\abs{e \cap \mathbold{I}}\bigr) = \sum_{v \in e} \bP(v \in \mathbold{I}) \geq k - 1.
    \]
    On the other hand, since $\mathbold{I}$ is an independent set and so does not contain $e$, $\abs{e \cap \mathbold{I}} \leq k - 1$ holds deterministically. It follows that, with probability 1, $\abs{e \cap \mathbold{I}} = k - 1$. Let $I$ be an independent set such that $\bP(\mathbold{I} = I) > 0$ and so $\abs{e \cap I} = k - 1$. In particular, $I$ contains exactly $k - 1$ vertices of $e$, $V(H) \setminus I$ contains exactly one vertex of $e$, and so both $I$ and $V(H) \setminus I$ do not contain $e$. Since $e$ was arbitrary and $I$ was chosen independently of $e$, this implies that $I$ and $V(H) \setminus I$ are both independent sets and hence $H$ is 2-colourable.
\end{proof}

A non-empty set of edges $\cC$ in a $k$-uniform multi-hypergraph is \defn{cyclic} if $\abs{V(\cC)} \leq (k - 1) \abs{\cC}$. This is a very weak form of cycle (for example, the edge-sets of Berge cycles, loose cycles, and tight cycles are all cyclic). A set of edges is a \defn{forest} if it contains no cyclic set.
The following lemma says that every cyclic set either contains a loose cycle or two edges meeting on at least two vertices.

\begin{lemma}\label{lem:cyclic}
    Let $\cC$ be an inclusion-minimal cyclic set of edges in a $k$-uniform multi-hypergraph. Then 
    \begin{itemize}
        \item either $\cC$ contains exactly two edges which meet on at least two vertices, or
        \item every $e \in \cC$ has $k - 2$ vertices that are contained in exactly one edge of $\cC$ \textup{(}the edge $e$\textup{)} and the other $2$ vertices are contained in exactly two edges of $\cC$.
    \end{itemize}
\end{lemma}

\begin{proof}
    The result is immediate if $\cC$ has size at most two.
    For each vertex $v \in V(\cC)$, write \defn{$\deg_{\cC}(v)$} for its degree: the number of edges of $\cC$ that contain $v$.
    Since $\cC$ is cyclic, we have $(k - 1) \abs{\cC} = \abs{V(\cC)} + a$ where $a \geq 0$. If some $e \in \cC$ has at most $a + 1$ vertices of degree at least 2, then $\cC \setminus \set{e}$ has at most $\abs{V(\cC)} - (k - a - 1) = (k - 1)(\abs{\cC} - 1)$ vertices and so is cyclic, contradicting the minimality of $\cC$. Thus every edge $e \in \cC$ has at least $a + 2$ vertices of degree at least two.
    Now,
    \begin{align*}
        a + \abs{\cC} & = k \abs{\cC} - \abs{V(\cC)} = \sum_{x \in V(\cC)} \bigl(\deg_{\cC}(x) - 1\bigr) \geq \sum_{\substack{x \in V(\cC) \colon \\ \deg_{\cC}(x) \geq 2}} \deg_{\cC}(x)/2 \\
        & = \sum_{e \in \cC} \abs{\set{x \in e \colon \deg_{\cC}(x) \geq 2}}/2 \geq \abs{\cC} \cdot (a + 2)/2,
    \end{align*}
    which rearranges to $a (\abs{\cC} - 2) \leq 0$. This implies that $a = 0$ and so there is equality throughout. That is, every edge contains exactly two vertices with degree at least two and these vertices have degree exactly two.
\end{proof}

The \defn{girth} of a (multi-)hypergraph $H$ is the size of the smallest cyclic $\cC \subset E(H)$. If $H$ is a forest, then we take its girth to be $\infty$. The existence of uniform hypergraphs with large girth and large chromatic number (in fact, large Hall ratio) is classical, going back to Erd\H{o}s and Hajnal~\cite{EH66}. We will need the following slight variant, in which the edges are partitioned into classes that each have small independence number. Its proof is standard and closely follows the Erd\H{o}s--Hajnal construction.

\begin{lemma}\label{lem:lowalpha_highgirth}
Let $r \geq 2$ and let $g$, $d$ be positive integers. For every $\eps > 0$, there exists an integer $n$ and edge-disjoint non-empty $r$-uniform hypergraphs $G_0, \dots, G_{d-1}$ on the common vertex-set $[n]$ such that $\alpha(G_i) \leq \varepsilon n$ for every $i$, and $G_0 \cup \dots \cup G_{d-1}$ is a hypergraph with girth at least $g$. 
\end{lemma}
\begin{proof}
We may assume that $\eps<1$. Take $N$ to be sufficiently large and set $p\coloneq N^{1/(2g)-(r-1)}$. Since $r\geq 2$, we may assume that $N$ is large enough so that $dp <1$. Independently assign each $e\in\binom{[N]}{r}$ a colour from the set $\set{0,\dots,d-1}$, each with probability $p$, and leave $e$ uncoloured with
probability $1-dp$. Then let $\mathbold{G}'_i$ be the hypergraph with vertex-set $[N]$ and all of the $i$-coloured edges. These edge-sets are pairwise disjoint by construction, so their union $\mathbold{G}' \coloneqq \mathbold{G}'_0 \cup \dots \cup \mathbold{G}'_{d-1}$ is a genuine hypergraph.

We first show that, with high probability, $\mathbold{G}'$ does not have many cyclic sets. For $2\leq s <g$, let $\mathbold{X}_s$ be the number of cyclic sets of size $s$ in $\mathbold{G}'$. By definition, every such set spans at most $(r-1)s$ vertices. Counting cycles by first choosing a set of $(r-1)s$ vertices and then choosing $s$ edges inside, the number of possible cyclic sets is at most 
\[
  \binom{N}{(r-1)s}\binom{\binom {(r-1)s}r}{s} \leq c_s N^{(r-1)s},
\]
for some constant $c_s$. A specific cyclic set of size $s$ occurs with probability $(dp)^s$, so $\bE[\mathbold{X}_s] \leq (dp)^s\cdot c_s N^{(r-1)s} = c_s d^s\, N^{s/(2g)}$.
Letting $\mathbold{X} = \sum_{2\leq s<g}\mathbold{X}_s$, we then have 
\[
\bE[\mathbold{X}] \leq \sum_{s=2}^{g-1}c_s d^sN^{s/(2g)} \leq \left(\sum_{s=2}^{g-1}c_s d^s\right) N^{(g-1)/(2g)} = o(N),
\]
so, by Markov's inequality, $\bP (\mathbold{X} \leq N/2) = 1 - o(1)$.

Next, we show that every $\mathbold{G}'_i$ has small independence number. Set $m \coloneqq \floor{\eps N/2}$, and assume that $N$ is large enough so that $m \geq r$. For a fixed $i$ and fixed $S\in\binom{[N]}m$, since each $r$-subset of $S$ is coloured $i$ independently with probability $p$, the probability that $S$ is an independent set in $\mathbold{G}'_i$ is 
$(1-p)^{\binom mr} 
\leq \exp(-p\binom mr)
 \leq \exp\bigl(-c\varepsilon^r N^{1+1/(2g)}\bigr)$ for some constant $c$. Taking a union bound over $d$ choices of $i$ and at most $2^N$ choices of $S$, the probability that $\alpha(\mathbold{G}'_i)\geq m$ for some $i$ is at most
\[ 
d\, 2^{N}\exp\bigl(-c\eps^r N^{1+1/(2g)}\bigr)
= d\exp\bigl(N(\log 2 - c\eps^r N^{1/(2g)})\bigr)
= o(1).
\]
Hence, $\alpha(\mathbold{G}'_i) < \eps N/2$ for every $i$ with probability $1-o(1)$.

We may therefore fix an outcome for which $\mathbold{X} \leq N/2$ and $\alpha(\mathbold{G}'_i) < \eps N/2$ for every $i$, which occurs with probability $1-o(1)$. Now, delete one vertex from each of the (at most $N/2$) cyclic sets in $\mathbold{G}'$ with size less than $g$. Let $V$ be the remaining vertices, and let $n \coloneq \abs{V} \geq N/2$. We claim that the induced sub-hypergraphs $G_i \coloneq \mathbold{G}'_i[V]$ have the properties that we want. Indeed, edge-disjointness is inherited and any cyclic set in the union $G \coloneqq G_0\cup\dots\cup G_{d-1}$ is a cyclic set in $\mathbold{G}'$, so $G$ has girth at least $g$. Also, we have $\alpha(G_i)\leq \alpha(\mathbold{G}'_i) < \eps N/2\leqslant \eps n$, and since $\varepsilon<1$ this gives $\alpha(G_i)<n$, so each $G_i$ has an edge. Relabelling $V$ as $[n]$ completes the proof.
\end{proof}

\section{Hypergraphs hung on rooted trees}\label{sec:hung}

The \defn{height} of a rooted tree is the length of the longest root-leaf path in the tree. A rooted tree is \defn{complete} if the lengths of all root-leaf paths are the same.

A rooted tree $T$ induces a partial order \defn{$\prec_T$} on $V(T)$ where $u \prec_T v$ if $u\neq v$ and the unique path in $T$ from the root to $v$ contains $u$ or, colloquially, if $v$ is a \defn{descendant} of $u$. We write \defn{$u \preceq_T v$} if $u \prec_T v$ or $u = v$.

For a rooted tree $T$ and vertex $u \in V(T)$, 
\[
\defn{\child(u)} \coloneqq \set{v \colon u \prec_T v, uv \in E(T)}
\]
denotes the set of children of $u$.

Let $T$ be a rooted tree. A \defn{$k$-uniform hypergraph hung on $T$} is a $k$-uniform hypergraph with vertex-set $V(T)$ and, for each edge $e$, there are vertices $u, v \in V(T)$ with $u \preceq_T v$ and $v$ not a leaf such that $u \in e$ and $e \setminus \set{u} \subset \child(v)$. For a $k$-uniform hypergraph $H$ hung on $T$ and vertices $u \preceq_T v$ where $v$ is not a leaf, we define the $(k - 1)$-uniform hypergraph \defn{$H_{u, v}$} to have vertex-set $\child(v)$ and edge-set
\[
\set{e \setminus \set{u} \colon u \in e \in E(H), \, e \setminus \set{u} \subset \child(v)}.
\]
Note that the edge-set of $H$ is $\set{\set{u} \cup e \colon u \preceq_T v, \, v \text{ not a leaf,} \, e \in E(H_{u, v})}$ and so we may specify $H$ by specifying all the $H_{u, v}$.

The following two lemmas control the Hall ratio and fractional chromatic number of hypergraphs that have been hung on rooted trees.

\begin{lemma}\label{lem:hung-Hall}
    Let $k \geq 3$ and $d$ be positive integers, let $T$ be a rooted tree of height $d$, and let $H$ be a $k$-uniform hypergraph hung on $T$. If $E(H) \neq \emptyset$ and, for every non-leaf vertex $v$, the multi-hypergraph $\bigcup_{u \preceq_T v} H_{u, v}$ has girth at least $d \cdot 2^{d + 1}$, then
    \[
        \rho(H) = \frac{k}{k - 1}.
    \]
\end{lemma}

We defer the proof of \cref{lem:hung-Hall} to \cref{sec:Hall}.

\begin{lemma}\label{lem:hung-chif}
    Let $\eps > 0$ and $d$ be a positive integer. Let $T$ be a complete rooted tree of height $d$ and let $H$ be a hypergraph hung on $T$ such that, for every $u \preceq_T v$ where $v$ is not a leaf, $\alpha(H_{u, v}) \leq \eps \cdot \abs{V(H_{u, v})}$. Then
    \[
        \chi_f(H) > \frac{1}{\epsilon + 1/(d + 1)}.
    \]
\end{lemma}

\begin{proof}
    By the definition of fractional chromatic number, there is some random independent set $\mathbold{I}$ of $H$ such that, for each vertex $v \in V(H)$,
    \[
        \bP(v \in \mathbold{I}) \geq \chi_f(H)^{-1}.
    \]
    Independently of $\mathbold{I}$, let $\mathbold{P} = \mathbold{v}_0 \mathbold{v}_1 \dots \mathbold{v}_d$ be a random root-leaf path in $T$ sampled as follows: $\mathbold{v}_0$ is the root vertex of $T$ and, for each $0 \leq r < d$, $\mathbold{v}_{r + 1}$ is a uniformly random child of $\mathbold{v}_r$.
    Fix some root-leaf path $P$. Note that
    \[
        \bE\bigl(\abs{\mathbold{I} \cap \mathbold{P}} \mid \mathbold{P} = P\bigr) = \sum_{v \in V(P)} \bP(v \in \mathbold{I}) \geq \abs{V(P)} \cdot \chi_f(H)^{-1} = (d + 1) \cdot \chi_f(H)^{-1}.
    \]
    Since $P$ was arbitrary, $\bE\bigl(\abs{\mathbold{I} \cap \mathbold{P}}\bigr) \geq (d + 1) \cdot \chi_f(H)^{-1}$.
    Fix an independent set $I$ in the support of $\mathbold{I}$ and a vertex $x \in I$ such that no ancestor of $x$ is in $I$. Let $r$ be the distance from the root to $x$.
    If $\mathbold{v}_r = x$, then no earlier vertex of $\mathbold{P}$ is in $I$ and so
    \[
        \bE\bigl(\abs{\mathbold{I} \cap \mathbold{P}} \mid \mathbold{I} = I, \, \mathbold{v}_r = x\bigr) = 1 + \sum_{s = r + 1}^d \bP(\mathbold{v}_s \in I \mid \mathbold{v}_r = x).
    \]
    Fix $s > r$. Note that $\mathbold{v}_{s - 1}$ is some $y \succeq_T x$. Now, conditioned on the events $\set{\mathbold{v}_r = x}$ and $\set{\mathbold{v}_{s - 1} = y}$, the vertex $\mathbold{v}_s$ is a uniformly random child of $y$ and so
    \[
        \bP(\mathbold{v}_s \in I \mid \mathbold{v}_r = x, \mathbold{v}_{s - 1} = y) = \frac{\abs{I \cap \child(y)}}{\abs{\child(y)}} = \frac{\abs{I \cap \child(y)}}{\abs{V(H_{x, y})}}.
    \]
    However, $I$ is an independent set and contains $x$ so $I \cap \child(y)$ is an independent set in $H_{x, y}$. Hence the final fraction is at most $\alpha(H_{x, y})/\abs{V(H_{x, y})} \leq \eps$. Thus $\bP(\mathbold{v}_s \in I \mid \mathbold{v}_r = x, \mathbold{v}_{s - 1} = y) \leq \eps$. Since $y$ was arbitrary, we have $\bP(\mathbold{v}_s \in I \mid \mathbold{v}_r = x) \leq \eps$. Substituting into the equation above, we have
    \[
        \bE\bigl(\abs{\mathbold{I} \cap \mathbold{P}} \mid \mathbold{I} = I, \, \mathbold{v}_r = x\bigr) \leq 1 + (d - r) \eps < 1 + (d + 1) \eps.
    \] 
    Now, either $\abs{I \cap \mathbold{P}} = 0$ or the first vertex, $x$, of $I$ that $\mathbold{P}$ contains has the properties that $x \in I$, no ancestor of $x$ is in $I$, and $\mathbold{v}_r = x$ for some $r$. Therefore $\bE\bigl(\abs{\mathbold{I} \cap \mathbold{P}} \mid \mathbold{I} = I\bigr) < 1 + (d + 1) \eps$. Since $I$ was arbitrary, $\bE\bigl(\abs{\mathbold{I} \cap \mathbold{P}}\bigr) < 1 + (d + 1) \eps$. Combining our upper and lower bounds,
    \[
        (d + 1) \cdot \chi_f(H)^{-1} \leq \bE\bigl(\abs{\mathbold{I} \cap \mathbold{P}}\bigr) < 1 + (d + 1) \eps,
    \]
    which rearranges to the required result.
\end{proof}

Our main theorem, \cref{thm:main}, follows readily from \cref{lem:hung-Hall,lem:hung-chif}.

\begin{proof}[Proof of \cref{thm:main}]
    Fix an integer $k \geq 3$ and some positive integer $K$. It suffices to show there is a $k$-uniform hypergraph $H$ such that $\rho(H) = k/(k - 1)$ and $\chi_f(H) > K$ (recall that the chromatic number is always at least the fractional chromatic number). Let $\eps > 0$ and $d$ be a positive integer such that $1/(\eps + 1/(d + 1)) \geq K$. 

    By \cref{lem:lowalpha_highgirth} applied with $r=k-1$, $g = d \cdot 2^{d+1}$, and the present $d, \eps$, there are non-empty $(k - 1)$-uniform hypergraphs $G_0, \dots, G_{d - 1}$ on the same vertex-set $[n]$ such that $\alpha(G_i) \leq \eps n$ for all $i$ and their union $G_0 \cup \dots \cup G_{d - 1}$ is a hypergraph with girth at least $d \cdot 2^{d + 1}$. Let $T$ be the complete $n$-ary tree of height $d$. We will specify a $k$-uniform hypergraph $H$ hung on $T$ by specifying all the $H_{u, v}$. For $u \preceq_T v$ where $v$ is not a leaf, we take $H_{u, v}$ to be $G_{\dist_T(u, v)}$.

    Now, $E(H) \neq \emptyset$ and, for each vertex $v$, the multi-hypergraph $\bigcup_{u \preceq_T v} H_{u, v}$ is a sub-hypergraph of $G_0 \cup \dots \cup G_{d - 1}$ and so has girth at least $d \cdot 2^{d + 1}$. Thus, \cref{lem:hung-Hall} implies that $\rho(H) = k/(k - 1)$. Next, for every pair $u \preceq_T v$ where $v$ is not a leaf, $H_{u, v}$ is one of the $G_i$ so satisfies $\alpha(H_{u, v}) \leq \eps n = \eps \cdot \abs{V(H_{u, v})}$. Thus \cref{lem:hung-chif} and our choice of $\eps$ and $d$ imply that
    \[
        \chi_f(H) > \frac{1}{\eps + 1/(d + 1)} \geq K,
    \]
    as required.
\end{proof}

\section{The Hall ratio}\label{sec:Hall}

A set $S \subset V(H)$ is a \defn{transversal} of $\cF$ if $e \cap S \neq \emptyset$ for every $e \in \cF$ and \defn{$\tau(\cF)$} denotes the size of the smallest transversal for $\cF$. Since the complement of an independent set is a transversal, we may work with transversals for convenience rather than deal directly with Hall ratio. 

\begin{lemma}\label{lem:Hall-transversal}
    Let $H$ be a $k$-uniform hypergraph \textup{(}$k \geq 2$\textup{)} with $E(H) \neq \emptyset$. Then $\rho(H) = k/(k - 1)$ if and only if, for every $\cF \subset E(H)$, $\tau(\cF) \leq \abs{V(\cF)}/k$.
\end{lemma}

\begin{proof}
    First suppose that $\rho(H) = k/(k - 1)$. Let $\cF \subset E(H)$. If $\cF = \emptyset$, then the assertion is immediate. Otherwise, applying $\rho(H) = k/(k - 1)$ to $X = V(\cF)$ gives an independent set $I \subset V(\cF)$ of size at least $(1 - 1/k) \abs{V(\cF)}$. Since $I$ is independent, no edge of $\cF$ is contained in $I$ and so $V(\cF) \setminus I$ is a transversal for $\cF$. In particular, $\tau(\cF) \leq \abs{V(\cF) \setminus I} \leq \abs{V(\cF)}/k$.

    Now suppose that $\tau(\cF) \leq \abs{V(\cF)}/k$ occurs for all $\cF \subset E(H)$. Let $X \subset V(H)$ be non-empty. Then there is some transversal $S$ for $E(H[X])$ of size at most $\abs{X}/k$. Since $S$ is a transversal, there are no edges contained in $X \setminus S$ and so $X \setminus S$ is an independent set in $H[X]$. Thus $\alpha(H[X]) \geq \abs{X \setminus S} \geq \abs{X} - \abs{S} \geq (1 - 1/k) \abs{X}$ and so $\abs{X}/\alpha(H[X]) \leq k/(k - 1)$. Since $X$ was arbitrary, $\rho(H) \leq k/(k - 1)$. On the other hand, $H$ contains at least one edge $e$ and so $\rho(H) \geq \abs{e}/\alpha(H[e]) = k/(k - 1)$.
\end{proof}

A set of edges $\cF$ is \defn{$\tau$-critical} if, for every $e \in \cF$, we have $\tau(\cF \setminus \set{e}) < \tau(\cF)$. It is \defn{connected} if there is no proper partition $\cF = \cF_1 \sqcup \cF_2$ such that $V(\cF_1) \cap V(\cF_2) = \emptyset$. Later in the proof we will need the following result of Stehl\'{i}k.

\begin{theorem}[{\cite[Thm.~2]{Stehlik2006}}]\label{thm:critical-connected}
    Every connected and $\tau$-critical set of edges $\cF$ satisfies $\abs{\cF} \geq 2 \tau(\cF) - 1$.
\end{theorem}
\subsection{Proof of Lemma \ref{lem:hung-Hall}}
Our goal for the rest of this section is to prove \cref{lem:hung-Hall}.
Let $T$ be a rooted tree of height $d$, and let $H$ be a $k$-uniform hypergraph hung on $T$ such that $E(H) \neq \emptyset$. For every non-leaf vertex $v$, the multi-hypergraph $\bigcup_{u \preceq_T v} H_{u, v}$ has girth at least $d \cdot 2^{d + 1}$.

Fix some $\cF \subset E(H)$. By \cref{lem:Hall-transversal}, it suffices to show that $\tau(\cF) \leq \abs{V(\cF)}/k$. We will do this by induction on $\abs{\cF}$ noting that it is immediate for $\cF = \emptyset$.

If $\cF$ is not connected, then let $\cF_1 \sqcup \cF_2$ be a proper partition of $\cF$ such that $V(\cF_1) \cap V(\cF_2) = \emptyset$. The union of a transversal for $\cF_1$ and a transversal for $\cF_2$ is a transversal for $\cF$ so, by induction,
\[
\tau(\cF) \leq \tau(\cF_1) + \tau(\cF_2) \leq \abs{V(\cF_1)}/k + \abs{V(\cF_2)}/k = \abs{V(\cF)}/k.
\]
If $\cF$ is not $\tau$-critical, then there is some $e \in \cF$ such that $\tau(\cF \setminus \set{e}) = \tau(\cF)$. Induction gives $\tau(\cF) = \tau(\cF \setminus \set{e}) \leq \abs{V(\cF \setminus \set{e})}/k \leq \abs{V(\cF)}/k$. Thus, from now on we may and will assume that $\cF$ is non-empty, connected, and $\tau$-critical.

For a non-leaf vertex $v$ of $T$, let
\begin{align*}
    \defn{\cF_v} & \coloneqq \set{e \in \cF \colon \abs{e \cap \child(v)} = k - 1}, \\
    \defn{\cF'_v} & \coloneqq \set{e \cap \child(v) \colon e \in \cF_v},
\end{align*}
where $\cF'_v$ may be a multi-hypergraph since multiple edges of $\cF_v$ could share the same $k - 1$ vertices in $\child(v)$ (in fact, we will shortly show in \cref{prop:forest} that this is not possible).
Note that $\cF'_v$ is a sub-hypergraph of $\bigcup_{u \preceq_T v} H_{u, v}$.
Furthermore, every edge of $\cF_v$ is of the form $e \cup \set{u}$ where $e \in \cF'_v$ and $u \preceq_T v$. For each $e \in \cF'_v$, let $\defn{u_e} \preceq_T v$ be such that $e \cup \set{u_e} \in \cF_v$.

\begin{proposition}\label{prop:forest}
    For each non-leaf $v$ of $T$, $\cF'_v$ is a forest.
\end{proposition}
\begin{poc}
    Suppose towards a contradiction that $\cF'_v$ is not a forest and let $\defn{\cC} \subset \cF'_v$ be an inclusion-minimal cyclic set. Note that $\abs{\cC} \geq d \cdot 2^{d + 1}$, since $\bigcup_{u \preceq_T v} H_{u, v}$ has girth at least $d \cdot 2^{d + 1}$.
    Since $\cF$ is $\tau$-critical, for every $e \in \cC$, there is a transversal \defn{$S_e$} for $\cF \setminus \set{e \cup \set{u_e}}$ of size less than $\tau(\cF)$ (in fact its size must be exactly $\tau(\cF) - 1$ but we will not need this). Since $v$ is not a leaf, $\set{u \colon u \preceq_T v}$ has size at most $d$ and so, by the pigeonhole principle, we may find $2d$ distinct edges $e_1, \dots, e_{2d} \in \cC$ such that $S_{e_i} \cap \set{u \colon u \preceq_T v}$ are the same for all $i$.

     Note that $S_{e_{2d}}$ must intersect $e_i \cup \set{u_{e_i}}$ for $1 \leq i \leq 2d - 1$. Now, if $u_{e_i} \in S_{e_{2d}}$, then, by the pigeonhole property, $u_{e_i} \in S_{e_i}$. But this then implies that $S_{e_i}$ is a transversal for $\cF$ of size less than $\tau(\cF)$ which is a contradiction. Thus $S_{e_{2d}}$ must intersect all of $e_1, \dots, e_{2d - 1}$. Since $\cC$ is inclusion-minimal, \cref{lem:cyclic} implies that every vertex is in at most two of $e_1, \dots, e_{2d - 1}$ and so, after possibly relabelling the $e_j$ we may find $d$ different vertices $x_1, \dots, x_d$ such that $x_i \in S_{e_{2d}} \cap e_i$ for $1 \leq i \leq d$.

     Let \defn{$T_{x_i}$} be the subtree of $T$ consisting of $x_i$ and all its descendants. 
     Consider the set obtained from $S_{e_i}$ by replacing $S_{e_i} \cap V(T_{x_i})$ with $S_{e_{2d}} \cap V(T_{x_i})$. We claim that this is a transversal for all of $\cF$. Indeed, every edge of $\cF$ that is disjoint from $V(T_{x_i})$ is hit since it is hit by $S_{e_i}$ (note $e_i$ meets $V(T_{x_i})$). Also, every edge contained entirely in $V(T_{x_i})$ is hit since it is hit by $S_{e_{2d}}$ ($e_{2d} \cup \set{u_{e_{2d}}}$ is not contained in $V(T_{x_i})$).
     An edge that intersects $V(T_{x_i})$ in exactly one element must have the form $e\cup\set{u_e}$ with $x_i\in e\in \cF'_v$, which is then covered by $x_i\in S_{e_{2d}}\cap V(T_{x_i})$.
     Every remaining edge $e$ contains one vertex $y \in \set{u \colon u \preceq_T v}$ and $k - 1$ in $V(T_{x_i})$. If $S_{e_{2d}}$ does not contain any of the $k - 1$ in $V(T_{x_i})$, then it contains $y$. By the pigeonhole property, $S_{e_i}$ also contains $y$ and so this new set is indeed a transversal. Therefore this new set has size at least $\tau(\cF)$. But $\abs{S_{e_i}} < \tau(\cF)$ and so we must have $\abs{S_{e_i} \cap V(T_{x_i})} \leq \abs{S_{e_{2d}} \cap V(T_{x_i})} - 1$.

     We now consider the set obtained from $S_{e_{2d}}$ by replacing each $S_{e_{2d}} \cap V(T_{x_i})$ by $S_{e_i} \cap V(T_{x_i})$ (for $1 \leq i \leq d$) and then adding all vertices of $\set{u \colon u \preceq_T v}$. Since the $x_i$ are distinct, each replacement reduces the size of the set by at least 1 so the resulting set at the end has size at most  $\abs{S_{e_{2d}}} - d + \abs{\set{u \colon u \preceq_T v}} \leq \abs{S_{e_{2d}}} < \tau(\cF)$. We now show that the final set is a transversal for $\cF$ which is the required contradiction. Every edge that meets at least one of the $V(T_{x_i})$ is either contained entirely inside one $V(T_{x_i})$ (and so hit by $S_{e_i}$) or has a vertex in $\set{u \colon u \preceq_T v}$ (and so is hit). Every other edge is entirely disjoint from all the $V(T_{x_i})$ so is either still hit by $S_{e_{2d}}$ or is the edge $e_{2d} \cup \set{u_{e_{2d}}}$ which is hit by $\set{u \colon u \preceq_T v}$. Thus the final set is indeed a transversal for $\cF$, as required.
\end{poc}

Let \defn{$R$} be the subtree of $T$ consisting of all paths from the root of $T$ to vertices of $V(\cF)$. Then $L \subset V(\cF) \subset V(R)$ where \defn{$L$} is the set of leaves of $R$. $R$ is still a rooted tree (with the partial order induced by $\prec_T$). Let \defn{$\child_R$} denote the child function, defined in \cref{sec:hung}, for the tree $R$. Note that $V(\cF'_v) \subset \child_R(v)$ for every non-leaf $v \in V(R)$. 
Fix a non-leaf $v \in V(R)$. If $\cF'_v = \emptyset$, then $\abs{\child_R(v)} - 1 \geq 0 \geq (k - 2) \abs{\cF'_v}$.
Otherwise, since $\cF'_v$ is a $(k - 1)$-uniform forest and so not cyclic,
\[
    \abs{\child_R(v)} \geq \abs{V(\cF'_v)} \geq (k - 2) \abs{\cF'_v} + 1.
\]
Thus, for every non-leaf $v \in V(R)$, $\abs{\child_R(v)} - 1 \geq (k - 2) \abs{\cF'_v}$.
Consider adding charges to vertices of $R$ as follows: for every non-leaf vertex $v$, a charge of $-1$ is given to $v$ and $+1$ to all of $v$'s children. The total charge given is $\sum_{\text{non-leaf } v \in V(R)} (\abs{\child_R(v)} - 1)$. On the other hand, the root of $R$ ends up with charge $-1$, the other internal vertices of $R$ end up with charge $0$, and the leaves of $R$ end up with charge $+1$. Thus
\[
(k - 2) \sum_{\text{non-leaf } v \in V(R)} \abs{\cF'_v} \leq \sum_{\text{non-leaf } v \in V(R)} (\abs{\child_R(v)}- 1) = \abs{L} - 1.
\]
Each set in $\cF'_v$ comes from an edge in $\cF_v$ and, furthermore, since $\cF'_v$ is a forest and so every pair of sets in $\cF'_v$ has intersection of size at most 1, different edges in $\cF_v$ give rise to different sets in $\cF'_v$. Thus $\abs{\cF'_v} = \abs{\cF_v}$ for every $v$. Furthermore, by the definition of a hung hypergraph, $(\cF_v \colon \text{non-leaf } v \in V(T))$ form a partition of $\cF$. Now, if $\cF_v \neq \emptyset$, then $V(\cF)$ meets $\child(v)$ and so $v$ is a non-leaf vertex of $R$. Therefore,
\[
\abs{L} - 1 \geq (k - 2) \sum_{\text{non-leaf } v \in V(R)} \abs{\cF_v} = (k - 2) \sum_{\text{non-leaf } v \in V(T)} \abs{\cF_v} = (k - 2) \abs{\cF}.
\]
Next notice that $V(\cF) \setminus L$ is a transversal of $\cF$: every edge of $\cF$ contains some vertex that is an ancestor of the other vertices in the edge and so is not in $L$. Therefore, $\tau(\cF) \leq \abs{V(\cF)} - \abs{L}$. Finally, recall, by \cref{thm:critical-connected}, that $\abs{\cF} \geq 2 \tau(\cF) - 1$. Putting everything together,
\begin{align*}
    \abs{V(\cF)} & \geq \tau(\cF) + \abs{L} \geq \tau(\cF) + (k - 2) \abs{\cF} + 1 \\
    & \geq \tau(\cF) + (k - 2)(2 \tau(\cF) - 1) + 1 \\
    & = k \tau(\cF) + (k - 3)(\tau(\cF) - 1) \\
    & \geq k \tau(\cF),
\end{align*}
where the final inequality uses $k \geq 3$ and that $\cF$ is non-empty. Thus $\tau(\cF) \leq \abs{V(\cF)}/k$ which completes the induction and so the proof of \cref{lem:hung-Hall}.

\paragraph{Acknowledgements.}
We thank Eoin Hurley for very many, very helpful discussions at the start of this project. 
The second author gratefully acknowledges Sergey Norin and Raphael Steiner for discussions on the problem during the 2026 MATRIX `Global structure and geometry of graphs' workshop.
We also thank Julia B{\"o}ttcher and Jozef Skokan for their helpful feedback on a draft of the paper.
We are grateful to the organisers of the UCL Combinatorics Workshop in July 2026 where this project was initiated. 

VB is supported by the EPSRC Studentship under the grant EP/Z534870/1. PN is supported by the ESRC Studentship under the grant ES/Y001869/1.

\paragraph{AI declaration.} AI was only used for a literature search and proofreading the paper.

{
\fontsize{11pt}{12pt}
\selectfont

\hypersetup{linkcolor={red!70!black}}
\setlength{\parskip}{2pt plus 0.3ex minus 0.3ex}

\bibliographystyle{Illingworth.bst}
\bibliography{references.bib}
}
\end{document}